\documentclass{amsart}
\usepackage{amsmath,latexsym,epsfig,graphics,amsfonts,amsthm,amssymb}

\usepackage[margin=1in]{geometry}
\usepackage{hyperref}

\newtheorem{theorem}{Theorem}
\newtheorem{lem}{Lemma}[section]
\newtheorem{prop}[lem]{Proposition}
\newtheorem{corollary}[lem]{Corollary}
\newtheorem{conjecture}{Conjecture}

\newcommand{\cC}{\mathcal{C}}

\newcommand{\Fq}{\mathbb{F}_q}

\newcommand{\SL}{\operatorname{SL}}

\newcommand{\mindeg}{\mathrm{mindeg}}
\newcommand{\minclass}{\mathrm{minclass}}

\renewcommand{\leq}{\leqslant}
\renewcommand{\geq}{\geqslant}

\begin{document}

\title{A generalization of a theorem of Rodgers and Saxl for simple groups of bounded rank}

\author{N. Gill}
\address{N. Gill, Department of Mathematics, University of South Wales, Treforest, CF37 1DL, U.K.}
\email{nick.gill@southwales.ac.uk}

\author{L. Pyber}
\address{A. R\'enyi Institute of Mathematics, Hungarian Academy of Sciences, P.O. Box 127, H-1364 Budapest}
\email{pyber@renyi.hu}

\author{E. Szab\'o}
\address{A. R\'enyi Institute of Mathematics, Hungarian Academy of Sciences, P.O. Box 127, H-1364 Budapest}
\email{endre@renyi.hu}



\begin{abstract}
We prove that if $G$ is a finite simple group of Lie type and $S_1,\dots, S_k$ are subsets  of $G$ satisfying $\prod_{i=1}^k|S_i|\geq|G|^c$ for some $c$ depending only on the rank of $G$, then there exist elements $g_1,\dots, g_k$ such that $G=(S_1)^{g_1}\cdots (S_k)^{g_k}$. This theorem generalizes an earlier theorem of the authors and Short. 

We also propose two conjectures that relate our result to one of Rodgers and Saxl pertaining to conjugacy classes in $\SL_n(q)$, as well as to the Product Decomposition Conjecture of Liebeck, Nikolov and Shalev.
\end{abstract}

\maketitle

\section{Introduction}

This note is inspired by two earlier results.
One of them is the following theorem of Rodgers and Saxl \cite{rodgers_saxl}:

\begin{theorem}\label{t: rs}
 Suppose that $\cC_1,\dots, \cC_k$ are conjugacy classes in $\SL_n(q)$ such that $\prod_{i=1}^k |\cC_i|\geq |\SL_n(q)|^{12}$. Then $\prod_{i=1}^k\cC_i = \SL_n(q)$.
\end{theorem}

The other is a result of Gill, Pyber, Short, Szab\'o \cite{GPSS}:
\begin{theorem}\label{t: gps}
Fix a positive integer $r$. There exists a constant $c=c(r)$ such that if $G$ is a finite simple group of Lie type of rank $r$ and $S$ is a subset of $G$ of size at least two then $G$ is a product of $N$ conjugates of $S$ for some $N\leq c \log|G|/ \log |S|$.
\end{theorem}

Our main result is a generalisation of Theorem~\ref{t: gps}
in the spirit to that of Rodgers and Saxl:

\begin{theorem}\label{t: aim1}
Let $G=G_r(q)$ be a finite simple group of Lie type of rank $r$. There exists $c=f(r)$ such that if $S_1,\dots, S_k$ are subsets of $G$ satisfying $\prod_{i=1}^k|S_i|\geq|G|^c$, then there exist elements $g_1,\dots, g_k$ such that $G=(S_1)^{g_1}\cdots (S_k)^{g_k}$.
\end{theorem}

Theorem~\ref{t: aim1} differs to that of Rodgers and Saxl in three important respects, two good, one not so good: First, our result pertains to all finite simple groups $G$ of Lie type. Second, our result does not just pertain to conjugacy classes, but to subsets of the group, provided we are free to take conjugates.

The third difference is a weak point: our result replaces the constant ``12''  in Theorem~\ref{t: rs} with an unspecified constant that depends on the rank of the group $G$. We conjecture that we should be able to do better, and not just for finite simple groups of Lie type, but for alternating groups as well:

\begin{conjecture}\label{c: ideal}
Let $G$ be a non-abelian finite simple group. There exists $c$ such that if $S_1,\dots, S_k$ are  subsets of $G$ satisfying $\prod_{i=1}^k|S_i|\geq|G|^c$, then there exist elements $g_1,\dots, g_k$ such that $G=(S_1)^{g_1}\cdots (S_k)^{g_k}$.
\end{conjecture}

Conjecture~\ref{c: ideal} seems out of reach at the moment. Indeed, it is a significant generalization of a conjecture that already exists in the literature -- the Product Decomposition Conjecture of Liebeck, Nikolov and Shalev \cite{lns2} -- and which already appears to be very challenging.

In light of the undoubted difficulty of proving Conjecture~\ref{c: ideal}, we propose a second, weaker conjecture. A proof of this conjecture, as well as being of interest in its own right, would represent a significant staging post in the pursuit of a proof of Conjecture~\ref{c: ideal}.

\begin{conjecture}\label{c: aim2}
Let $G$ be a non-abelian finite simple group. There exists $c$ such that if $S_1,\dots, S_k$ are normal subsets of $G$ satisfying $\prod_{i=1}^k|S_i|\geq|G|^c$, then $G=S_1\cdots S_k$.
\end{conjecture}

Note that $S$ is a \emph{normal} subset of the group $G$ if it is invariant under conjugation by elements of $G$; in other words, $S$ is a union of conjugacy classes of $G$. Conjecture~\ref{c: aim2} is a generalization of an important theorem of Liebeck and Shalev \cite{lieshal}.

\subsection{Structure of the paper}
In \S\ref{s: background} we give the necessary background results used to prove Theorem~\ref{t: aim1}. In \S\S\ref{s: medium sets}, \ref{s: large sets} and \ref{s: small sets}, we give partial results towards a proof of Theorem~\ref{t: aim1}, depending on the size of the sets $S_1,\dots, S_k$: we use different techniques if these sets are ``small'', ``medium-sized'' or ``large''. Finally, in \S\ref{s: proof} we prove Theorem~\ref{t: aim1}.

\subsection{Acknowledgement}
\label{sec:acknowledgement}
N.~Gill was supported by EPSRC grant EP/N010957/1. L.~Pyber and
E. Szab\'o were supported by the National Research, Development and
Innovation Office (NKFIH) Grant K115799. E.~Szab\'o was also supported
by the NKFIH Grant K120697. The project leading to this application
has received funding from the
European Research Council (ERC) under the European Union's Horizon
2020 research and innovation programme (grant agreement No 741420).

\section{Necessary background}\label{s: background}

We will use a theorem of Petridis \cite[Theorem~1.7]{petridis}:

\begin{theorem}\label{t: petridis}
 Let $A$ and $B$ be finite sets in a group $G$. Suppose that
 \begin{enumerate}
  \item $|AB|\leq \alpha|A|$.
  \item $|AbB|\leq \beta|A|$ for all $b\in B$.
  \item $|A|\leq \gamma|B|.$
 \end{enumerate}
Then there exists $S\subseteq A$ such that for all $h>1$,
\[
 |SB^h|\leq \alpha^{8h-9}\beta^{h-1}\gamma^{4h-5}|S|.
\]
where $B^h$ denotes the product of $h$ copies of $B$.
\end{theorem}

From here on $G$ is a finite group. Let $\minclass(G)$ denote the size of the smallest nontrivial conjugacy class in $G$, and let $\mindeg(G)$ denote the dimension of the smallest nontrivial complex irreducible representation of $G$.

As observed in \cite{npy}, a result of Gowers \cite{gowers} implies the following.

\begin{prop}\label{p: gowers}
Let $G$ be a finite group and let $k=\mindeg(G)$. Take $A,B,C\subseteq G$ such that
$|A|\cdot |B|\cdot |C| \geq \frac{|G|^3}{k}.$ Then $G=ABC$.
\end{prop}

The following two results give useful facts about simple groups of Lie type. Note that, if we write $G_r(q)$ for a simple group of Lie type of rank $r$ over $\Fq$, then there are multiple conventions for the definition of $r$ and the definition of $q$. We have stated the following results very conservatively -- they are valid for whichever standard definition of these two parameters one cares to take (and this also explains the difference in the statement of the first, from that which appears in \cite{GPSS}).

The first result follows for the classical groups from \cite[Table 5.2.A]{kl} for the classical results (taking into account the corrections listed in \cite{mv}), and for the exceptional groups from \cite{v1, v2, v3}.

The second result is proved using the lower bounds on projective representations given by Landazuri and Seitz \cite{landseitz} (taking into account the corrections listed in
\cite[Table 5.3.A]{kl}).

\begin{prop}\label{p: conjugacy size}
Let $G=G_r(q)$ be a simple group of Lie type of rank $r$ over $\Fq$, the finite field of order $q$. We have $q^{r/2}\leq \minclass(G)<|G|\leq q^{8r^2}$.
\end{prop}

\begin{prop}\label{p: landseitz}
Let $G=G_r(q)$ be a simple group of Lie type of rank $r$ over $\Fq$, the finite field of order $q$. Let $k=\mindeg(G)$. Then $|G|<k^{8r^2}$.
\end{prop}

Note that Propositions~\ref{p: gowers} and \ref{p: landseitz} imply that if $A,B,C$ are subsets of $G=G_r(q)$ with $|A|,|B|,|C|>|G|^{1-\frac{1}{24r^2}}$, then $G=ABC$.

The next result was obtained independently in \cite{guralnick-kantor} and \cite{stein}. The subsequent corollary is an easy consequence, and can be found in \cite{GPSS}. Note that, by the {\it translate} of a set $S$ in a group $G$, we mean a set of form $Sg:=\{sg \mid s\in S\}$ where $g$ is some element of $G$.

\begin{prop}\label{p: stein}
Each finite simple group $G$ is $\frac32$-generated; that is, for any nontrivial element $g$ of $G$ there exists $h$ in $G$ such that $\langle g, h\rangle=G$.
\end{prop}

\begin{corollary}\label{c: 32}
Let $G$ be a finite simple group and let $S$ be a subset of $G$ of size at least two. Then some translate of $S$ generates $G$.
\end{corollary}

Finally we need the Product Theorem, proved independently in \cite{BGT} and \cite{PS}.

\begin{theorem}\label{t: generating sets}
Fix a positive integer $r$. There exists a positive constant $\eta=\eta(r)$ such
that, for $G$ a finite simple group of Lie type of rank $r$ and $S$ a
generating set of $G$,  either $S^3=G$ or $|S^3|\geq |S|^{1+\eta}.$
\end{theorem}

\section{Medium-sized sets}\label{s: medium sets}

\begin{lem}\label{l: two sets}
Fix $r>0$. There exists $\varepsilon>0$ such that if $A$ and $B$ are subsets of $G=G_r(q)$, a finite simple group of Lie type of rank $r$, with $2\leq |B|\leq |A|$, then one of the following holds:
\begin{enumerate}
 \item $|A|\geq |B|^{1+\varepsilon}$;
 \item there exists $g\in G$ such that $|A\cdot B^g|\geq |A|\cdot |B|^{\varepsilon}$;
 \item $|A|\geq |G|^{1/26}\cdot |B|^{25/26}$;
 \item there exists $g\in G$ such that $|A\cdot B^g|\geq |G|^{1/25} \cdot |A|^{24/25}$.
\end{enumerate}
\end{lem}
\begin{proof}
Appealing to Corollary~\ref{c: 32}, let $B_0=Bx$ be a translate of $B$ that generates $G$.
Define $\gamma=|A|/|B|$, $\alpha=|AB|/|A|$ and $\beta=\max\{|A\cdot B^{b^{-1}}|/|A| \mid b\in B\}$.  We apply Theorem~\ref{t: petridis} with $h=3$ to obtain that there exists $S\subset A$ such that
\[
 |S\cdot B_0^3| \leq \alpha^{15}\beta^2\gamma^7|S|.
\]
This implies in particular that
\begin{equation}\label{e: petridis}
 |B_0^3|\leq \alpha^{15}\beta^2\gamma^7|A|.
\end{equation}
Now, since $B_0$ generates $G$, Theorem~\ref{t: generating sets} gives two possibilities for $B_0^3$.

First, suppose that $|B_0^3|\geq |B|^{1+\eta}$. We obtain that
\[
 |B|^\eta \leq \alpha^{15}\beta^2\gamma^8.
\]
We conclude that at least one of $\alpha, \beta$ or $\gamma$ is greater than or equal to $|B|^{\eta/25}$, and this implies that either $|A|\geq |B|^{1+\eta/25}$, or else there exists $g\in G$ such that $|A\cdot B^g|\geq |A|\cdot |B|^{\eta/25}$. Taking $\varepsilon=\eta/25$, we obtain one of the first two listed possibilities.

The second possibility is that $B_0^3=G$. Now \eqref{e: petridis} implies that
\[
 |G|\leq \alpha^{15}\beta^2\gamma^7|A|.
\]
We  conclude that at least one of $\alpha, \beta$ or $\gamma$ is greater than or equal to $(|G|/|A|)^{1/25}$, and some simple rearranging yields the final two listed possibilities. 
\end{proof}

\begin{lem}\label{l: two sets 2}
Fix $r>0$ and $0<\delta<1$. There exists $\eta=f(r,\delta)>0$ such that if $A$ and $B$ are subsets of $G=G_r(q)$, a finite simple group of Lie type of rank $r$, with $2\leq |B|\leq |A|$, then one of the following holds:
\begin{enumerate}
 \item $|A|\geq |B|^{1+\eta}$;
 \item there exists $g\in G$ such that $|A\cdot B^g|\geq |A|\cdot |B|^{\eta}$ and there exists $h\in G$ such that $|B^h \cdot A|\geq |A|\cdot |B|^{\eta}$;
 \item $|B|\geq |G|^\delta$.
\end{enumerate}
 \end{lem}
\begin{proof}
Let $\varepsilon$ be the positive number whose existence is guaranteed by Lemma~\ref{l: two sets}. Define $\eta = \min\{\frac{1-\delta}{26\delta}, \varepsilon\}$. Then $\delta \leq \frac{1}{1+26\eta}$, and we apply Lemma~\ref{l: two sets}. If the third option of that lemma holds, then we obtain that either $|A|\geq |B|^{1+\eta}$ or else
\[
|G|^{1/26}\cdot |B|^{25/26}<|B|^{1+\eta} 
\]
and, rearranging, we get that $|B|\geq |G|^{\frac{1}{1+26\eta}}\geq|G|^\delta$, as required. 

Similarly, if the fourth option of Lemma~\ref{l: two sets} holds, then we obtain that either $|A\cdot B^g|\geq |A|\cdot |B|^{\eta}$ or else
\[
 |G|^{1/25}|A|^{24/25}<|A|\cdot |B|^\eta.
\]
in which case we obtain that $|G|<|A|\cdot |B|^{25\eta}$. Then either $|A|\geq |B|^{1+\eta}$ (and so (1) holds), or else we obtain that $|G|<|B|^{1+26\eta}$ and we obtain (3) as before. 

If the first option of Lemma~\ref{l: two sets} holds, then the first option holds here. Finally, suppose that the second option of Lemma~\ref{l: two sets} holds. We obtain immediately that the first part of option (2) holds here. To see that the second part holds, observe that
\[
 |B^h\cdot A| = |A^{-1}\cdot (B^{-1})^h|.
\]
Now we can apply Lemma~\ref{l: two sets} to the two sets $A^{-1}$ and $B^{-1}$. If the first, third or fourth option holds, then the argument given above implies that the item (1) or (3) holds here for $A^{-1}$ and $B^{-1}$, hence also for $A$ and $B$. On the other hand if the second option holds, then we obtain the second part of item (2) here, and we are done.
\end{proof}

\begin{lem}\label{l: medium sets}
 Fix $0<\zeta<\delta<1$ and $r$ a positive integer. Then there exists $c=f(\zeta, \delta, r)>0$ such that if $S_1,\dots, S_t\subset G$, where
 \begin{enumerate}
 \item $G$ is a finite simple group of Lie type of rank $r$;
  \item $|S_i|\geq |G|^\zeta$;
  \item $\prod\limits_{i=1}^t |S_i| \geq |G|^{c}$;
 \end{enumerate}
then there exist elements $g_1,\dots, g_t\in G$ and positive integers $k_1, k_2,k_3$ such that $t=k_1+k_2+k_3$ and
\[\min\{|T_1|,|T_2|,|T_3|\}\geq |G|^\delta\]
where $T_1=S_1^{g_1}\cdots S_{k_1}^{g_{k_1}}$, 
$T_2=S_{k_1+1}^{g_{k_1+1}}\cdots S_{k_1+k_2}^{g_{k_1+k_2}}$ and 
$T_3=S_{k_1+k_2+1}^{g_{k_1+k_2+1}}\cdots S_{t}^{g_t}$.
\end{lem}

Note that no attempt is made in the subsequent proof to optimise $c$.

\begin{proof}
Let  $\eta=f(r,\delta)$ be the constant whose existence is guaranteed by Lemma~\ref{l: two sets 2}. Let $S_1,\dots, S_t$ be subsets of $G$ satisfying condition (2).

Let $\kappa=\log|S_i|/\log|G|$ where $S_i$ is the smallest set in $S_1,\dots, S_t$. By supposition, $\kappa\geq \zeta$.
We will apply Lemma~\ref{l: two sets 2} a number of times so as to produce a new family of larger sets $S_1',\dots, S_{\lfloor\frac{t}{2}\rfloor}'$: For each even $i$ between $2$ and $t$, let $A$ be the larger of $S_i$ and $S_{i-1}$, and let $B$ be the smaller. Lemma~\ref{l: two sets 2} gives three possibilities.

If the first possibility holds, then  $|A|\geq|G|^{\kappa(1+\eta)}$, and we let $S_{i/2}'=S_{i-1}S_i$. If the second possibility holds and $A=S_{i-1}$, then we choose $g$ so that $|A\cdot B^g|$ is as large as possible, and we set $S_{i/2}'=S_{i-1}S_i^g$; if the second possibility holds and $A=S_i$, then we choose $h$ so that $|B^h\cdot A|$ is as large as possible, and we set $S_{i/2}'=S_{i-1}^hS_i$. Notice that in both of these cases we end with $|S_{i/2}'| \geq |G|^{\kappa(1+\eta)}$.
If the third possibility holds, then $|B|\geq|G|^\delta$ and we set $S_{i/2}'=S_{i-1}\cdot S_i$.

Observe that there are $\lfloor t/2\rfloor\geq t/3$ sets in our new family, and that the minimum size of a set in the new family is at least $|G|^{\min\{\kappa(1+\eta),\delta\}}$. 

We repeat this process as long as $\kappa<\delta$. We must choose $c$ to ensure that we end with at least $3$ sets in our final family: all of these, by construction, will have size at least $|G|^\delta$, and the result follows.
Note first, that the minimum size of a set in the family produced after $i$ iterations is at least $|G|^{\zeta(1+\eta)^i}$. Now $\zeta(1+\eta)^i\geq \delta$ if and only if 
\[
 i\geq I:=\frac{\log \delta - \log \zeta}{\log(1+\eta)}.
\]
On the other hand, after each iteration, the number of sets diminishes by at most a third, so if we start with at least $3^{I+1}$ sets, then we will definitely end with at least $3$ sets, as required. To ensure that we start with this number of sets, then we can take $c=3^{I+1}$, and we are done.
\end{proof}

\section{Large sets}\label{s: large sets}

To deal with large sets, we will use ``the Gowers trick'', Proposition~\ref{p: gowers}. When combined with our work on medium-sized sets, we obtain the conclusion that we need.

\begin{prop}\label{p: large sets}
  Fix $0<\zeta<1$ and $r$ a positive integer. Then there exists $c=f(r,\zeta)>0$ such that if $S_1,\dots, S_t\subset G$, where
 \begin{enumerate}
 \item $G$ is a finite simple group of Lie type of rank $r$;
  \item $|S_i|\geq |G|^\zeta$;
  \item $\prod\limits_{i=1}^t|S_i| \geq |G|^{c}$;
 \end{enumerate}
then there exist elements $g_1,\dots, g_t\in G$ such that 
\[
 S_1^{g_1}\cdots S_t^{g_t}=G.
\]
\end{prop}
\begin{proof}
Set $\delta= {1-\frac{1}{24r^2}}$ and apply Lemma~\ref{l: medium sets}. The resulting three sets $T_1,T_2,T_3$ satisfy the property that $\min\{|T_1|, |T_2|, |T_3|\}\geq |G|^\delta$ and Propositions~\ref{p: gowers} and \ref{p: landseitz} imply that $T_1\cdot T_2\cdot T_3=G$ (see the remark after Proposition~\ref{p: landseitz}). But, given the definition of the sets $T_1,T_2$ and $T_3$, the desired conclusion follows immediately.
\end{proof}

\section{Small sets}\label{s: small sets}

In this section we use a variant of the ``greedy lemma'' argument of \cite{GPSS}. First we need an easy little lemma.

\begin{lem}\label{l: base subset prelim}
If $A$ and $B$ are finite subsets of a group $G$ then
 \[
|AB||A^{-1}A\cap BB^{-1}|\geqslant |A||B|.
 \]
\end{lem}

Note that a similar result is stated by Helfgott in \cite[Lemma~2.2]{helfgott3}.

\begin{proof}
Let $m=|AB|$. Choose elements $a_1,\dots,a_m$ of $A$ and $b_1,\dots,b_m$ of $B$ such that $AB=\{a_1b_1,\dots,a_mb_m\}$. Let $A^{-1}A\cap BB^{-1}=\{x_1,\dots,x_n\}$. Consider the map 
\[
\Theta: AB\times(A^{-1}A\cap BB^{-1}) \to G\times G,\qquad (a_ib_i,x_j)\mapsto (a_ix_j^{-1},x_jb_i).
\]
The map $\Theta$ is injective, because, given an element
$(a_ix_j^{-1},x_jb_i)$ we can recover the element
$a_ib_i=a_ix_j^{-1}x_jb_i$. Since the elements $a_1,\dots, a_m$ and
$b_1,\dots, b_m$ are fixed
and each element of $AB$ has a unique expression of the form $a_kb_k$
we recover the elements $a_i$ and $b_i$, along with the element $x_j$, and injectivity follows. Therefore
\[
|AB||A^{-1}A\cap BB^{-1}|=|AB\times(A^{-1}A\cap BB^{-1})|=|\Theta(AB\times(A^{-1}A\cap BB^{-1}))|.
\]
We complete the proof by establishing that $A\times B$ is in the image of $\Theta$. Given $(a,b)$ in $A\times B$ we can choose $i$ such that $ab=a_ib_i$. Therefore $a^{-1}a_i=bb_i^{-1}$; this element belongs to $A^{-1}A\cap BB^{-1}$, and hence is equal to $x_j$, for some $j$. Therefore $(a,b)=\Theta(a_ib_i,x_j)$, as required.
\end{proof}

\begin{lem}\label{l: B}
Given subsets $A$ and $B$ of a finite group $G$ we have  
\[
\sum_{C\in\mathcal{C}(G)} \frac{|A\cap C||B\cap C|}{|C|} = \frac{1}{|B^G|}\sum_{B'\in B^G} |A\cap B'|,
\]
where $\mathcal{C}(G)$ is the set of conjugacy classes in $G$,
and $B^G$ denotes the set of $G$-conjugates of $B$.
\end{lem}
\begin{proof}
First observe that
\[
\sum_{C\in\mathcal{C}(G)} \frac{|A\cap C||B\cap C|}{|C|} = \sum_{C\in\mathcal{C}(G)} \sum_{a\in A\cap C} \frac{|B\cap C|}{|C|}= \sum_{a\in A}\frac{|B\cap a^G|}{|a^G|}.
\]
Now,
\[
\bigcup_{B'\in B^G} \{(a',B')\,:a'\in a^G,a'\in B'\}=\bigcup_{a'\in a^G} \{(a',B')\,:B'\in B^G,a'\in B'\},
\]
and comparing cardinalities gives $|B^G||B\cap a^G|=|a^G|\sum_{B'\in B^G}1_{B'}(a)$ where we define
\[
 1_{B'}(a):=\left\{\begin{array}{ll}
                    1,& a\in B',\\
                    0,&\textrm{otherwise}.
                   \end{array}\right.
\]
It follows that 
\begin{align*}
\sum_{a\in A}\frac{|B\cap a^G|}{|a^G|} &= \frac{1}{|B^G|}\sum_{a\in A}\sum_{B'\in B^G} 1_{B'}(a)\\
 &= \frac{1}{|B^G|}\sum_{B'\in B^G}\sum_{a\in A} 1_{B'}(a)\\
 &= \frac{1}{|B^G|}\sum_{B'\in B^G}|A\cap B'|,
\end{align*}
as required.
\end{proof}

\begin{prop}\label{p: small sets}
  Suppose $A$ and $B$ are subsets of a finite group $G$.
  Suppose, in addition, that $|A|,|B|< (\minclass(G))^{1/4}$. Then there exists $g\in G$ such that $|A\cdot B^g|=|A|\cdot |B|$.
\end{prop}
\begin{proof}
 Suppose that we cannot find such a $g$. This implies that, for every $B_*$ conjugate to $B$, $|AB_*|<|A|\cdot |B|$. Now Lemma~\ref{l: base subset prelim} yields
 \[
|A|\cdot |B||A^{-1}A\cap B_*B_*^{-1}|>|AB||A^{-1}A\cap B_*B_*^{-1}|\geqslant |A||B|.
 \]
We obtain that $|A^{-1}A\cap B_*B_*^{-1}|\geq 2$. As before, let $\mathcal{C}(G)$ be the set of conjugacy classes in $G$, and let $\mathcal{C}^*(G)$ be the set of non-trivial conjugacy classes in $G$. Lemma~\ref{l: B} implies that
 \begin{align*}
&\sum\limits_{C\in \mathcal{C}(G)} \frac{|C\cap A^{-1}A| \cdot |C\cap BB^{-1}|}{|C|} = \frac{1}{|(BB^{-1})^G|}\sum_{X\in (BB^{-1})^G} |A^{-1}A\cap X|\geqslant 2 \\
\Longrightarrow & \sum\limits_{C\in \mathcal{C}^*(G)} \frac{|C\cap A^{-1}A| \cdot |C\cap BB^{-1}|}{|C|} + 1\geqslant 2 \\
\Longrightarrow &  \sum\limits_{C\in \mathcal{C}^*(G)} \frac{|C\cap A^{-1}A|}{|C|}\geqslant \frac{1}{|BB^{-1}|}.
\end{align*}
In particular we obtain that $|A^{-1}A|\geq \min\limits_{C\in\mathcal{C}^*(G)}\frac{|C|}{|BB^{-1}|}$. Now, since $|A^{-1}A|\leq |A|^2$ and $|BB^{-1}|\leq |B|^2$, the result follows.
\end{proof}

\section{A proof of Theorem~\ref{t: aim1}}\label{s: proof}

\begin{proof}[Proof of Theorem~\ref{t: aim1}]
 Let $\zeta=\frac{1}{32r}$, and note that Proposition~\ref{p: conjugacy size} implies that $|G|^\zeta<(\minclass(G))^{1/4}$. Let $c_0$ be the constant whose existence is guaranteed by Proposition~\ref{p: large sets}. We define $c=2c_0+\zeta$; observe that, since $\zeta$ depends only on $r$, $c$ also depends only on $r$. 

Suppose, first of all, that there exists $i$ such that $|S_i|, |S_{i+1}|\leq |G|^\zeta$. Then Proposition~\ref{p: small sets} implies that there exists $g$ such that $|S_i\cdot S_{i+1}^g|=|S_i|\cdot |S_{i+1}|$. Thus we replace $S_i$ and $S_{i+1}$ with this product; this does not affect the ordering of the sets, nor does it affect the product of the cardinalities of the sets. We repeat this process until there are no ``adjacent'' sets of cardinality less than $|G|^\zeta$.
 
If $k$ is even, then, for every even $i$ between $1$ and $k$ we replace $S_{i-1}$ and $S_i$ by the product of the two. This results in a family of sets with the same ordering, all of which have order at least $|G|^\zeta$, and for which the product of cardinalities is at least $|G|^{c_0+\zeta/2}$. Now Proposition~\ref{p: large sets} implies the result.

If $k$ is odd and $|S_k|\geq|G|^\zeta$, then, for every even $i$ between $1$ and $k$, we replace $S_{i-1}$ and $S_i$ by the product of the two and we retain $S_k$. We obtain a family with the same properties as in the previous paragraph and, once again, Proposition~\ref{p: large sets} implies the result.

If $k$ is odd and $|S_k|<|G|^\zeta$, then for every even $i$ between $1$ and $k-3$ we replace $S_{i-1}$ and $S_i$ by the product of the two; we also replace $S_{k-2}, S_{k-1}$ and $S_k$ by the product of the three. This results in a family of sets with the same ordering, all of which have order at least $|G|^\zeta$, and for which the product of cardinalities is at least $|G|^{c_0}$. Now Proposition~\ref{p: large sets} implies the result and we are done.
\end{proof}

\bibliographystyle{plain}
\bibliography{paper6}

\end{document}